\documentclass[12pt]{amsart}

\usepackage[margin=1.15in]{geometry}

\usepackage{amsmath,amscd,amssymb,amsfonts,latexsym,wasysym, mathrsfs,mathtools,hhline,color, bm}
\usepackage[all, cmtip]{xy}
\usepackage{comment}

\usepackage{comment}

\usepackage{mathptmx}

\usepackage{url}

\definecolor{hot}{RGB}{65,105,225}

\usepackage[pagebackref=true,colorlinks=true, linkcolor=hot ,  citecolor=hot, urlcolor=hot]{hyperref}

\usepackage{ textcomp }
\usepackage{ tipa }

\usepackage{graphicx,enumerate}

\usepackage{enumitem}

\theoremstyle{plain}
\newtheorem{theorem}{Theorem}[section]
\newtheorem{proposition}[theorem]{Proposition}

\newtheorem{lm}[theorem]{Lemma}

\newtheorem{corollary}[theorem]{Corollary}

\newtheorem{thrm}[theorem]{Theorem}

\theoremstyle{definition}

\newtheorem{definition}[theorem]{Definition}

\newtheorem{remark}[theorem]{Remark}

\newtheorem{ex}[theorem]{Example}
\newtheorem*{ex*}{Example}

\numberwithin{equation}{section}

\def\be{\begin{equation}}
\def\ee{\end{equation}}

\def\bt{\begin{thrm}}
\def\et{\end{thrm}}

\def\bc{\begin{cor}}
\def\ec{\end{cor}}

\def\br{\begin{rmk}}
\def\er{\end{rmk}}

\def\bp{\begin{prop}}
\def\ep{\end{prop}}

\def\bl{\begin{lm}}
\def\el{\end{lm}}

\def\bex{\begin{ex}}
\def\eex{\end{ex}}

\def\bd{\begin{defn}}
\def\ed{\end{defn}}

\newcommand\sM{{\mathcal M}}
\newcommand\sW{{\mathcal W}}

\newcommand\sS{\mathscr{S}}

\newcommand\cW{\mathcal{W}}

\renewcommand\d{\mathrm{d}}

\def\bR{\mathbf{R}}

\DeclareMathOperator{\Jac}{Jac}

\DeclareMathOperator{\codim}{codim}              
\DeclareMathOperator{\reg}{reg}                  
                  
\DeclareMathOperator{\Eu}{\mathrm{Eu}}

\DeclareMathOperator{\Sing}{Sing}   

\DeclareMathOperator{\mult}{mult}  
\newcommand{\lk}{{\mathbb C \rm{lk}}}

\DeclareMathOperator{\rank}{rank}

\DeclareMathOperator{\ord}{ord}

\def\bC{\mathbb{C}}

\def\cM{\mathcal{M}}

\def\bP{\mathbb{P}}

\def\cO{\mathcal{O}}
\def\lra{\longrightarrow}

\def\bZ{\mathbb{Z}}

\def\C{\mathbb{C}}

\title[Euclidean distance degree and limit points in a Morsification]{Euclidean distance degree and limit points in a Morsification}
\author{Lauren\c{t}iu Maxim}
\address{L. Maxim: Department of Mathematics,  University of Wisconsin-Madison,  480 Lincoln Drive, Madison WI 53706-1388, USA.}
\email {maxim@math.wisc.edu}

\author[M. Tib\u{a}r]{Mihai Tib\u{a}r}
\address{M. Tib\u{a}r: Universit\' e de  Lille, CNRS, UMR 8524 -- Laboratoire Paul Painlev\'e, F-59000 Lille, France}  
\email {mihai-marius.tibar@univ-lille.fr}



\date{\today}

\begin{document}

\begin{abstract}  
Motivated by finding an effective way to compute the algebraic complexity of the nearest point problem for algebraic models,
we introduce an efficient method for detecting the limit points of the stratified Morse trajectories in a small perturbation of 
any polynomial function on a complex affine variety.  We compute the multiplicities of these limit points in terms of
vanishing cycles. In the case of functions with only isolated stratified singularities, we express the local multiplicities in terms of  polar intersection numbers. 
\end{abstract}

\maketitle

\section{Introduction}

The Euclidean distance degree (ED degree) \cite{DHOST} of an affine algebraic variety $X\subset \mathbb{C}^N$ is defined as the number of critical points of the squared Euclidean distance function $d_u(x) := \sum_{i=1}^N(x_i-u_i)^2$ on the smooth locus $X_{\reg}$ of $X$, for {\it generic} $u=(u_1,\dots,u_N)$. It is denoted by ${\rm EDdeg}(X)$. If $X$ is the complexification of a real algebraic model $X_{\bR}$, the ED degree of $X$ measures the algebraic complexity of the nearest point problem for $X_{\bR}$. 
The Euclidean distance degree was introduced in \cite{DHOST}, and has since been extensively studied in areas like computer vision, 
biology, 
chemical reaction networks, 
engineering, 
numerical algebraic geometry, 
data science, 
etc.

In \cite{MRW2018}, the authors gave a purely topological interpretation of the Euclidean distance degree of a complex affine variety, expressed in terms of the weighted Euler characteristic of MacPherson's {local Euler obstruction function} \cite{M74}. As a concrete application, they solved the  
{\it $n$-view conjecture} of \cite{DHOST}, which was motivated  by  the
{\it triangulation  problem} in computer vision \cite{SSN}, and was the main reason for the introduction of the Euclidean distance degree.


Typically, results on Euclidean distance degrees and nearest point problems have a hypothesis requiring genericity of the data point $u$, e.g., see \cite{AH2018,AH,BD2015, MR3789441,MR3996403, MR3563098,MRW2018}.  However, in many practical situations, the data is not generic, e.g., when the data is sparse, 
or one may consider the Eckart-Young problem (i.e., finding the closest bounded rank matrix) with low rank data.

Results of \cite{MRW5}, some of which are overviewed in Section \ref{sec:2}, 
allow one to handle situations when the data is not generic by observing the ``limiting'' behavior of critical sets obtained for generic choices of data. 
More precisely, by adding some {\it noise} $\epsilon \in \C^N$ to an arbitrary data point $u$, one is back in the generic situation, and the limiting behavior of critical points of $d_{u+t\epsilon}$ on $X_{\rm reg}$ for $t \in \C^*$ (with $\vert t \vert$ very small), as $t$ approaches the origin of $\C$, yields valuable information about the initial nearest point problem. 

Notice that, in the notations of the previous paragraph, one can write $d_{u+t\epsilon}(x)=d_{u}(x)- tl(x) + c$,
with $l(x)=2 \sum_{i=1}^N  \epsilon_i  x_i$ and $c$ is a constant with respect to $x$.
So the critical points of $d_{u+t\epsilon}$ coincide with those of $d_u-tl$. 
Moreover, since  $\epsilon$ is generic, $l$ is a generic linear function.
This observation places us at the origins of the following Morsification procedure considered in particular in the recent paper \cite{MRW5}. 

Let $f\colon \C^N \to \C$ be a polynomial function, and let $l\colon \bC^N \to \bC$ be a linear function. Let $X \subset \C^N$ be a possibly singular closed irreducible subvariety such that $f$ is not constant on $X$, and restrict $f$ and $l$ to $X$. If the linear function $l$ is general enough, then the function $f_t:=f-tl$ is a holomorphic Morse function on $X_{\rm reg}$ (i.e., it has only non-degenerate isolated critical points) for all but finitely many $t \in \C$. One is typically  interested in studying the limiting behavior of the set of critical points of $f_t\vert_{X_{\reg}}$ as $t$ approaches $0 \in \bC$. 

 In the case $X=\C^N$ and when $f\colon \C^N \to \C$ is a polynomial function with only isolated singularities $P_1,\ldots,P_r$, a solution to the above problem is provided by the classical Morsification picture, as proved by Brieskorn in  \cite[Appendix]{Br}. More precisely, if $\mu_i$ is the {Milnor number} of $f$ at $P_i$ (cf. \cite{Mi}), then in a small neighborhood of $P_i$ the function $f_t$ has $\mu_i$ Morse critical points which, as $t$ approaches $0$, collide together at $P_i$. 
 
In the general situation, for $X$ and $f$ with arbitrary singularities, the paper \cite{MRW5} studies this limiting behavior of the set of critical points of $f_t\vert_{X_{\reg}}$ by using constructible functions and vanishing cycle techniques, together with work of Ginsburg \cite{G} on characteristic cycles. 
In order to understand the significance of formula \eqref{eq_maini} displayed below, which is the main result of \cite{MRW5}, we need to explain its ingredients, as follows.
Let $\Eu_X$ be the local Euler obstruction function on $X$, regarded as a constructible function on $\bC^N$ by extension by zero. Then $\Eu_X$ is constructible on $\bC^N$ for any Whitney stratification of $X$, to which one adds the open stratum $\bC^N \setminus X$. 
 We endow $X$ with a Whitney stratification $\mathscr W$ with finitely many strata, with respect to which the stratified singular set of $f$ is defined as $\Sing_{\mathscr W}f:=\bigcup_{V\in \mathscr W} \Sing (f\vert_{V})$. The set $\Sing_{\mathscr W}f$ is then a closed set in $X$ distributed in a finite number of critical fibers of $f$. 
 If $c \in \bC$ is a critical value of $f\colon X \to \bC$ and $\Phi_{f-c}$ denotes the corresponding vanishing cycle functor of constructible functions, then $\Phi_{f-c}(\Eu_X)$ is a constructible function supported on $\Sigma_c:=\{f=c\} \cap \Sing_\mathscr{W} f$. Each $\Sigma_c$ gets an induced Whitney stratification witnessing the constructibility of $\Phi_{f-c}(\Eu_X)$. One can thus refine $\mathscr W$ to a Whitney stratification $\mathscr S$ of $X$ adapted to the functions $\Phi_{f-c}(\Eu_X)$, for all critical values $c$ of $f$, and we will occasionally use the notation  $\Sing_{\mathscr S}f$ instead of $\Sing_{\mathscr W}f$ whenever strata in each $\Sigma_c$ need to be taken into account.(Of course, as sets, $\Sing_{\mathscr S}f =\Sing_{\mathscr W}f$.)
 Following \cite{MRW5}, we introduce uniquely determined integers $n_V$ for each stratum $V \subset \Sing_\mathscr{S} f$
so that:
\be\label{fcfi}
\sum_{c \in \bC} \Phi_{f-c}(\Eu_X)=\sum_{V \subset \Sing_\mathscr{S} f} (-1)^{\codim V-1}\cdot n_V \cdot \Eu_{\overline{V}}.
\ee
It can be shown that $n_V \geq 0$, for any $V \subset \Sing_\mathscr{S} f$.
With these notations, the main result of \cite{MRW5} can now be stated as follows (cf. Definition \ref{def21} for the notion of the limit of a family of sets of points):
\begin{equation}\label{eq_maini}
\lim_{t\to 0}\Sing(f_t|_{X_{\reg}})=\sum_{V \subset \Sing_\mathscr{S} f}n_{V}\cdot\Sing(l|_{V}).
\end{equation}
The left hand side of \eqref{eq_maini} does not take into account the points of $\Sing(f_t|_{X_{\reg}})$ which escape at infinity as $t \to 0$, that is, the singular points of $f$ on $X_{\reg}$ which are outside a sufficiently large ball centered at the origin for sufficiently small $t$.

Let us now get back to the calculation of the ED degree of the affine variety $X \subset \bC^N$. In this case, we let $f=d_u$ be the squared Euclidean distance function as above, but with $u \in \bC^N$ {\it arbitrary}. For $l$ a general linear function, if no points of $\Sing(f_t|_{X_{\reg}})$ go to infinity as $t \to 0\in \bC$, a formula for ${\rm EDdeg}(X)$ was given in \cite{MRW5} in terms of the multiplicities $n_V$ as:
\be\label{edi}
{\rm EDdeg}(X)=\sum_{V \subset \Sing_\mathscr{S} f}n_{V}\cdot \#\Sing(l|_{V})= \# \Sing(f_t|_{X_{\reg}}),
\ee
with $\#$ denoting the cardinality of a set.

Formulas \eqref{eq_maini} and  \eqref{edi} emphasize the need for computability of the multiplicities $n_V$, which measure the asymptotics of singularities in a Morse perturbation of $f$. 

\medskip

The main goal of this paper is to produce an efficient computational method  for the asymptotics of Morse singularities in a generic degeneration of a polynomial function. Specifically, we aim to detect the limit points of the stratified Morse trajectories and to compute the multiplicities of these limit points.

We start from observing that the general polar curve $\Gamma_{\mathscr W}(l,f)$, an algebraic set of dimension $1$, contains all the trajectories of the stratified Morse points of $f_{t}$ when $t\to 0$ (including those which escape to infinity).
We show that the multiplicities $n_V$ occurring in formula \eqref{eq_maini}, for $V$ a stratum of $\sS$ contained in a critical fiber of $f$,   may be interpreted as the number of Morse points of $f_{t}$ on $X_{\reg}$  which converge,  when $t\to 0$, to any of the points in $V \cap \Gamma_{\mathscr W}(l,f)$. In this way, the numbers  $n_V$ become  ``local multiplicities'' at those well determined points.   The advantage of this viewpoint on the localisation of the multiplicities is that the global polar curve is computable by equations.

Nevertheless, an explicit calculation of the local multiplicities $n_V$ of formula \eqref{eq_maini} in terms of the geometry and topology of the pair $(X,f)$ is difficult in general.  One of the main technical goals of this note is to overcome this difficulty by using new strategies. In Section \ref{sec:3}, we use the integers
$$\mu_V=\Phi_{f-c}(\Eu_X)(V),$$
i.e., the values
of the constructible function $\Phi_{f-c}(\Eu_X)$ along critical strata $V \subset \Sing_\mathscr{S} f$ of $f$, to
produce the following  formula for the multiplicities $n_V$ 
(see Theorem \ref{nvgen}): 
\begin{theorem}\label{nvgeni}
Let $X \subset \bC^N$ be a  complex affine variety, and $f\colon X \to \bC$ the restriction to $X$ of a polynomial function. Then, for any critical value $c$ of $f$, the multiplicities $n_V$ for singular strata $V \subset f^{-1}(c)$ are given by:
\begin{equation}\label{nvgi}
n_V=(-1)^{\codim V -1} \{ \mu_V - \sum_{ \{S \mid V \subset \overline{S} \setminus S \} } \chi_c(\lk_{\overline S}(V)) \cdot \mu_S \},
\end{equation}
where:
\begin{enumerate}
\item[(i)] the summation is over singular strata $S$ in $f^{-1}(c)$, different from $V$, which contain $V$ in their closure.
\item[(ii)] $\chi_c(\lk_{\overline S}(V))$ is the compactly supported Euler characteristic of the complex link of $V$ in $\bar S$, for a pair of singular strata $(V,S)$ in $f^{-1}(c)$, with $V \subset \overline{S}\setminus S$.
\end{enumerate} 
\end{theorem}
Formula \eqref{nvgi} becomes quite explicit if  $X$ is smooth, since in this case 
$\mu_V=\chi(\widetilde{H}^*(F_V;\bC))$ 
is just  the Euler characteristic of the reduced cohomology of the Milnor fiber $F_V$ of the hypersurface $\{f=c\}$ at some point in $V$  (cf Remark \ref{xsm}).
In particular, if $X$ is smooth and  $f\colon X\to \C$ has only isolated singularities, then at such an isolated singular point $P$ we have $n_P=\mu_P$, the Milnor number of $f$ at $P$, as predicted by the classical morsification picture. 

The integers $\mu_V$ appearing in \eqref{nvgi} are also used in Theorem \ref{morsenumber} to give a new interpretation for  the number of Morse critical points on the regular part of $X$ in a Morsification of $f$. More precisely, by combining   results of \cite{MRW5} and \cite{STV1} with standard properties of MacPherson's local Euler obstruction function, we prove the following result (see Theorem \ref{morsenumber}).
\begin{theorem}\label{morsenumberi}
The number of Morse critical points on $X_{\reg}$ in a generic deformation $f_t:=f-tl\colon X \to \bC$ of $f$ is given by:
\be\label{mnui}
\# \Sing(f_t|_{X_{\reg}})=m_\infty + (-1)^{\dim X-1}  \sum_{c \in \bC} \left( \sum_{V \subset f^{-1}(c) \cap  \Sing_\mathscr{S} f} \chi(V \setminus V \cap H_t) \cdot \mu_V \right) ,
\ee
where $m_\infty $ is the number of points of $\Sing(f_t|_{X_{\reg}})$ that escape to infinity as $t \to 0$,  the first sum is over the critical values $c$ of $f$, and $H_t:=l^{-1}(t)$ is a generic hyperplane.
\end{theorem}


When no critical points of $f_t|_{X_{\reg}}$ escape at infinity as $t\to 0$, 
one obtains from \eqref{edi} and \eqref{mnui} a new and explicit formula for the ED degree.

\smallskip

In Section \ref{s:isoldecomp} we focus on functions $f\colon X \to \C$ with only stratified isolated singularities, 
and we indicate a computation of the multiplicities $n_{P}$, for any point $P \in \Sing_{\mathscr W} f$,  in terms of polar intersection numbers. More precisely, 
one has the following:

\begin{theorem}\label{t:main1i}
Let $X\subset \bC^{N}$ be any affine variety and let $f\colon X\to \bC$ be the restriction of a polynomial function on $\bC^{N}$ such that it has only  isolated singularities on $X$ with respect to the stratification $\sW$. Then, for any general linear function $l$,  we have the following equality at any point $P \in \Sing_{\sW} f$:
\begin{equation}\label{eq:main1i}
n_P= \mult_{P}(\Gamma_{X_{\reg}}(l,f), f^{-1}(f(P)) - \mult_{P}(\Gamma_{X_{reg}}(l,f), l^{-1}(l(P)),
  \end{equation}
  where 
   $$\Gamma_{X_{\reg}}(l,f) :=  \overline{ \Sing ((l,f)\vert_{X_{\reg}}) \setminus \Sing (f\vert_{X_{\reg}})} $$ 
   is the polar locus of the Zariski-open stratum $X_{\reg}$ of $\sW$ (cf Definition \ref{d:polarcurve}).
\end{theorem}
In view of formula \eqref{edi}, we thus obtain in this setup a geometric and algebraically computable expression for the ED degree in terms of polar multiplicities. 




Several examples are discussed in Section \ref{s:ex}.

\medskip 

{\bf Acknowledgements.}  This work was started during a ``Research in Pairs'' program at the Mathematisch Forshungsinstitut Oberwolfach, to which the authors would like to thank for hospitality and ideal working conditions. L. Maxim was partially supported by the Simons Foundation, CNRS and the MPIM-Bonn. M. Tib\u{a}r acknowledges partial support by the Labex CEMPI grant (ANR-11-LABX-0007-01).


\section{Preliminaries}
\label{sec:2}





Let $X\subset \bC^N$ be an irreducible complex affine variety and let $f\colon \bC^N\to \bC$ be a polynomial function whose restriction to $X$ is nonconstant. 
Denote by $T^*_X \bC^N$ the conormal space of $X$ in the cotangent bundle $T^*\bC^N$, and recall that via the characteristic cycle functor one can define vanishing cycles of $f$ as a functor on conic lagrangian cycles in $T^*\bC^N$. 
Upon fixing a stratification $\sS$ of $X$ as in the Introduction, one has an  equality of lagrangian cycles
\begin{equation}\label{eq_nii}
\sum_{c\in \bC}\,^p\Phi_{f-c}([T_X^*\bC^N])=\sum_{V \subset \Sing_\mathscr{S} f}n_{V} \cdot [T^*_{\overline{V}}\bC^N],
\end{equation}
where $^p\Phi_{f-c}$ denotes the perverse vanishing cycle functor, and only strata $V$ in the stratified singular locus of $f$ which are contained in the critical fibers contribute to the sum on the right-hand side. 
 The sum on the left-hand side of \eqref{eq_nii} is finite since $f|_X$ has only finitely many critical values in the stratified sense, and $^p\Phi_{f-c}([T_X^*\bC^N])=0$ if $c$ is not a critical value. Moreover, the coefficients $n_V$ are nonnegative integers (see, e.g., \cite{Massey1}).

\medskip

 The multiplicities $n_V \in \bZ_{\geq 0}$  of \eqref{eq_nii} turn out to be particularly important for counting Morse critical points of a Morsification $f_t$ of $f$ on the regular locus $X_{\rm reg}$ of $X$, which converge to points in $X$ as $t \to 0$.
Let us explain this, by first recalling a few definitions from \cite{MRW5}. We refer to \cite[Section 2]{MRW5}  for more details on background and terminology.

By a {\it set of points} is meant a finite set endowed with a multiplicity function. This means that, after fixing a ground (i.e., ambient) set $S$, a set of points $\cM$ of $S$ is given by a function $\cM: S\to \mathbb{Z}_{\geq 0}$ such that $\cM(x)=0$ for all but finitely many $x\in S$. The value $\cM(x)$ is called the {\it multiplicity} of $\cM$ at $x$. 

\begin{definition}\label{def21}
Fix a Hausdorff space $S$ as the ground set, and let $\cM_t$ be a family of sets of points of $S$, parametrized by $t\in D^*$, with $D^*$ a punctured disc centered at the origin. The {\it limit} of $\cM_t$ as $t\to 0$, denoted by $\lim_{t\to 0}\cM_t$, is defined as the set of points given by:
\[
(  \lim_{t\to 0}\cM_t  )  (x)\coloneqq \varprojlim_{U}  \lim_{t\to 0}\sum_{y\in U}\cM_t(y),
\]
where $\varprojlim_U$ denotes the inverse limit over all open neighborhood of $x$. 
\end{definition}

All limits considered in \cite{MRW5} are of algebraic nature, and it is easy to see that in this case the limit $\lim_{t\to 0}\sM_t$ exists. Using these notations, the main result of \cite{MRW5} generalizes the classical Morsification picture \cite{Br} as follows:

\begin{theorem}\label{thm_main}\cite[Theorem 1.3]{MRW5}
Let $X\subset \bC^N$ be a complex affine variety and let $f\colon \bC^N\to \bC$ be a polynomial function whose restriction to $X$ is nonconstant. 
Let $l\colon \bC^N\to \bC$ be a general linear function, and consider $f_t\coloneqq f-tl$, $t \in \bC$, as a function on $X$. 
The limit of the critical points of $f_t$ satisfies
\begin{equation}\label{eq_main}
\lim_{t\to 0}\Sing(f_t|_{X_{\reg}})=\sum_{V \subset \Sing_\mathscr{S} f}n_{V}\cdot\Sing(l|_{V})
\end{equation}
where the symbol $\Sing$ denotes the set of critical points, $\Sing_\mathscr{S} f$ is the stratified singular locus of $f$ with respect to the Whitney stratification $\mathscr{S}$ of $X$ chosen as above, and the numbers $n_V$ are uniquely determined by formula \eqref{eq_nii}. 
\end{theorem}

It should be noted that the left side of equation (\ref{eq_main}) does {not} take into account the points of $\Sing(f_t|_{X_{\reg}})$ that ``go to infinity'' as $t\to 0$. Here we say that \emph{no points of $\Sing(f_t|_{X_{\reg}})$ go to infinity} if 
$\bigcup_{0<t\leq \epsilon}\Sing(f_t|_{X_{\reg}})$
 is bounded in $\bC^N$ for any sufficiently small $\epsilon \in \mathbb{R}_{>0}$. 
So, by (\ref{eq_main}), no points of $\Sing(f_t|_{X_{\reg}})$ go to infinity if and only if
\be\label{count} \# \Sing(f_t|_{X_{\reg}})=\sum_{V \subset \Sing_\mathscr{S} f}n_V \cdot \#\Sing(l|_{V}),\ee
for general $t$, where $\#$ denotes the cardinality of a set.

When coupled  with  \cite[Theorem 1.3]{MRW2018}, 
formula \eqref{eq_main} yields the following computation of the Euclidean distance (ED) degree of an affine variety $X$, 
in terms of the 
 critical sets of a general linear function $l$ on the strata $V$,
and where the distance function is considered with respect to an \emph{arbitrary} data point $u$.
 
\begin{corollary}\label{cor:ED}\cite[Corollary 1.9]{MRW5} \ \\
Let  $(u_1,\dots,u_N)\in\mathbb{C}^N$ be a data point, let $f=\sum_{i=1}^N(x_i-u_i)^2$, and let  $X\subset\mathbb{C}^N$ be an algebraic variety.
If no points of $\Sing(f-tl|_{X_{\reg}})$ escape to infinity as $t\to 0$, 
then the Euclidean distance degree of $X$ is given by
$${\rm EDdeg}(X)=\sum_{V \subset \Sing_\mathscr{S} f}n_{V}\cdot \#\Sing(l|_{V}).$$
\end{corollary}

\begin{remark}\label{r25}
Under the assumptions and notations of Corollary \ref{cor:ED}, if we assume moreover that $f$ has only isolated stratified critical points $P_1, \ldots, P_r$, then
\be\label{ediso}
{\rm EDdeg}(X)=\sum_{i=1}^r n_{P_i}.
\ee
For concrete applications of this formula, coupled with our polar interpretation of the multiplicities $n_{P_i}$ from Section \ref{s:polarcomp}, see Section \ref{s:ex}.
\end{remark}




\section{Local and global numbers of Morse points
in a generic deformation}\label{s:locglobMorse}\label{sec:3}


In this section, we obtain a general formula for the multiplicities $n_V$ appearing in equation \eqref{eq_nii} and Theorem \ref{thm_main}, 
 for arbitrary $X$ and for $f\colon X \to \C$ with arbitrary singularities; see Theorem \ref{nvgen}. In Section \ref{ss:numberMorse}, we show how formula \eqref{eq_main} together with results from \cite{STV1} can be used to compute  the number of Morse critical points on the regular part of $X$ in a Morsification of $f$.
  

\subsection{The multiplicities $n_V$}\label{mnv}

First note that, by using the characteristic cycle 
isomorphism $$CC:CF(\bC^N) \overset{\sim}{\lra} LC(T^*\bC^N)$$ between the free abelian group $CF(\bC^N)$ of algebraically constructible functions on $\bC^N$ and the free abelian group $LC(T^*\bC^N)$ spanned by conic lagrangian cycles in $T^*\bC^N$ (where we work with objects supported on $X$), equation \eqref{eq_nii} amounts to expressing (up to a sign), for each critical value $c$ of $f|_X$, the constructible function $\varphi_{f-c}(\Eu_X)$ in terms of the basis of local Euler obstruction functions $\Eu_{\overline{V}}$ corresponding to closures of singular strata $V$ in the special fiber $\{f=c\}$. Here, $\Eu_X$ denotes the local Euler obstruction function of $X$, introduced by MacPherson in \cite{M74}. Specifically, using the identity
\be
CC(\Eu_Z)=(-1)^{\dim Z} [T^*_Z\bC^N]
\ee
for any irreducible closed subvariety $Z \subset \bC^N$, together with the ``perverse shift'' $^p\Phi_{f-c}=\Phi_{f-c}[-1]$ (which in the language of constructible functions translates into multiplication by $-1$), equation \eqref{eq_nii} translates into the following identity of constructible functions:
\be\label{fcf}
\sum_{c \in \bC} \Phi_{f-c}(\Eu_X)=\sum_{V \subset \Sing_\mathscr{S} f} (-1)^{\codim V -1}\cdot n_V \cdot \Eu_{\overline{V}}.
\ee
In general, an explicit calculation of the coefficients $n_V$ in terms of the geometry and topology of the pair $(X,f)$ is difficult, and this is in fact the main technical goal of this note. 
In Example \ref{str_iso} and Section \ref{s:isoldecomp}, we consider the already interesting, but still challenging case when $f\colon X \to \C$ has only stratified isolated singularities and $X$ is arbitrary.  A general abstract formula for each $n_V$ is obtained in Theorem \ref{nvgen} below; this formula become quite explicit,  in terms of the topology of the pair $(X,f)$, in the case when $X$ is smooth and $f$ has arbitrary singularities, see Remark \ref{xsm} at the end of this subsection.


\begin{ex}\label{str_iso} This case was already considered in \cite[Example 1.2]{MRW5}, we include it here since it pertains to the results of Section \ref{s:isoldecomp}.  
	Let $X\subset \bC^N$ be an arbitrary complex affine variety, and assume that the polynomial function $f\colon X\to \bC$ is nonconstant and has only isolated stratified critical points $P_1, \ldots, P_r$. Then formula \eqref{fcf} becomes:
	\be\label{unu}\sum_{c \in \bC} \Phi_{f-c}(\Eu_X)=(-1)^{\dim X-1} \sum_{i=1}^r n_{P_i}\cdot \Eu_{P_i},\ee
	with $n_{P_i}$ obtained by evaluating \eqref{unu} at the point $P_i$, namely:
	\be\label{niis} n_{P_i}=(-1)^{\dim X-1}\Phi_{f-f(P_i)}(\Eu_X)(P_i) 
	= :(-1)^{\dim X} \Eu_{f-f(P_i)} (X,P_i),
	\ee
	 with the last term  denoting the {\it relative Euler obstruction} of the function $f-f(P_i)$ on $X$ at $P_i$ (as introduced in \cite{BMPS}).
	Theorem \ref{thm_main} specializes in this case to  
	\be\label{sta}
\lim_{t\to 0}\Sing(f_t|_{X_{\reg}})=\sum_{i=1}^r n_{P_i} \cdot P_i ,
\ee
with $n_{P_i}$ computed as in formula \eqref{niis}.
If $X$ is smooth, then $\Eu_X=1_X$, and formula \eqref{niis} yields that $n_{P_i}$ equals the Milnor number $\mu_{P_i}$ of $f$ at $P_i$, as predicted by the classical Morsification picture \cite{Br}. 
\end{ex}

Even in the particular case of Example \ref{str_iso}, it is difficult in general (when $X$ is arbitrarily singular) to evaluate the multiplicity $n_{P_i}$ by using formula \eqref{niis}. In Example \ref{str_isos}, we perform such an explicit (and lengthy) calculation in a very special situation, when $X$ has  itself  only an isolated singularity at the singularity of $f$. 
That particular calculation will be subsumed by
the general result discussed in Section \ref{s:isoldecomp}, which is, however, derived by other methods.


The next result gives a general abstract formula for the multiplicities $n_V$ appearing in formula \eqref{eq_nii} and Theorem \ref{thm_main}. For
 any singular stratum $V \subset f^{-1}(c)$, let us set the notation (and it should be clear for the context which critical fiber is considered):
\be\label{def} \mu_V:=\Phi_{f-c}(\Eu_X)(V).\ee

\begin{theorem}\label{nvgen}
Let $X \subset \bC^N$ be a  complex affine variety, and $f\colon X \to \bC$ the (non-constant) restriction to $X$ of a polynomial function. Then, for any critical value $c$ of $f$, the multiplicities $n_V$ for singular strata $V \subset f^{-1}(c)$, appearing in \eqref{eq_nii} or \eqref{fcf}, are computed by the formula:
\begin{equation}\label{nvg}
n_V=(-1)^{\codim V -1} \{ \mu_V - \sum_{ \{S \mid V \subset \overline{S} \setminus S \} } \chi_c(\lk_{\overline S}(V)) \cdot \mu_S \},
\end{equation}
where:
\begin{enumerate}
\item[(i)] the summation is over singular strata $S\in {\mathscr S}$ contained in $f^{-1}(c)$, different from $V$, and which contain $V$ in their closure.
\item[(ii)] $\chi_c(\lk_{\overline S}(V))$ is the compactly supported Euler characteristic of the complex link of $V$ in $\bar S$, for a pair of singular strata $(V,S)$ in $f^{-1}(c)$, with $V \subset \overline{S}\setminus S$.
\end{enumerate} 
\end{theorem}

\begin{proof}
For the expert reader, \eqref{nvg} is an application of the formula for the characteristic cycle of a bounded constructible complex (see, e.g., \cite[Sect.8.2]{G}). In the following, in order to show how the numbers $\mu_V$ come up, we will recast \eqref{nvg} in the language of constructible functions, as in \cite[Sect.5.0.3]{Sch}. 

The integers $n_V$ are defined by the identity \eqref{eq_nii}, whose localization at the critical fiber $f^{-1}(c)$ is obtained, up to the factor $(-1)^{\dim X -1}$, by applying the characteristic cycle functor $CC$ to 
\be\label{fcfb}
\Phi_{f-c}(\Eu_X)=\sum_{V \subset \Sigma_c } (-1)^{\codim V-1}\cdot n_V \cdot \Eu_{\overline{V}}.
\ee
Here, $\Sigma_c:=f^{-1}(c) \cap  \Sing_\mathscr{S} f$, and the summation in \eqref{fcfb} is over the singular strata $V$ contained in $f^{-1}(c)$.

Now recall that if $\varphi$ is a constructible function on a closed subvariety $Z \subseteq X$, which is a closed union of strata of $\mathscr{S}$, then the characteristic cycle  of $\varphi$, that is, 
\[ CC(\varphi)=\sum_{V \subset Z} c_V(\varphi) \cdot [T^*_{\overline{V}}\bC^N], \]
is determined by the microlocal multiplities $c_V(\varphi) \in \bZ$ attached to strata $V$ in $Z$. These multiplicities are computed by:
\be\label{mm} c_V(\varphi)=\sum_{V \subset \overline{S}} e_{V,S} \cdot \varphi(S),\ee
where $\varphi(S)$ is the value of $\varphi$ at a point in the stratum $S$, and
\begin{itemize}
\item[(i)] $e_{V,V}=(-1)^{\dim V}$,
\item[(ii)] $e_{V,S}=(-1)^{\dim V} \cdot \chi_c(\lk_{\overline S}(V))$, for any $V \subset \overline{S} \setminus S$.
\end{itemize}

 Formula \eqref{nvg} follows now by applying \eqref{mm} to the constructible function
 $$\varphi:=(-1)^{\dim X -1} \cdot  \Phi_{f-c}(\Eu_X)$$ on $\Sigma_c$, by noting that $n_V=c_V(\varphi)$ and $\mu_S=(-1)^{\dim X -1} \cdot \varphi(S)$.
 \end{proof}

\begin{remark}\label{27}
 Let us untangle some more details about  the integer $\mu_V$ defined in \eqref{def}. For a stratum $V \subset \Sigma_c$, we have:
\be\label{def2}\mu_V:=\Phi_{f-c}(\Eu_X)(V)=\Psi_{f-c}(\Eu_X)(V)-\Eu_X(V),\ee
with $\Psi_{f-c}:CF(X) \to CF(f^{-1}(c))$ the nearby cycle functor for the function $f-c$. Both terms on the right-hand side of \eqref{def2} can be expressed as a weighted sum over strata $S$ of $X$ different from $V$ and containing $V$ in their closure. Indeed, for a positive dimensional stratum $V$,  let $N_V$ be a general linear subspace of $\bC^N$ of codimension $\dim V$, which meets $V$ transversally at a point $v\in V$, and set $X_V:=X \cap N_V$ and $f_V:=f\vert_{X_V}$. Then, as in  \cite[Section 3]{BLS},
\be\label{def4} \Eu_X(V)=\Eu_{X_V}(v),\ee
and, similarly, by \cite[Lemma 4.3.4]{Sch},
\be\label{def5} \Psi_{f-c}(\Eu_X)(V)=\Psi_{f_V-c}(\Eu_{X_V})(v).\ee
Next, by applying \cite[Theorem 3.1]{BLS} (see also \cite{S}) to the germ $(X_V,v)$ with the induced stratification (in which $v$ is a zero-dimensional stratum), and using \eqref{def2}, \eqref{def4} and \eqref{def5}, we can write\footnote{Note that the positive dimensional strata in a stratification of $(X_V,v)$ are of the form $S \cap N_V$, with $S$ strata of $X$, and hence in one-to-one correspondence with the strata $S$ of $X$ different from $V$, and containing $V$ in their closure.}:
\be\label{def3}
\mu_V=\sum_{\{ S \mid V \subset \overline{S}\setminus S\}} 
\left[    
 \chi(f^{-1}(c') \cap B_{\epsilon}(v) \cap S) 
 - \chi(l^{-1}(\delta) \cap B_{\epsilon'}(v) \cap S) 
\right] \cdot \Eu_X(S)
\ee
with $\epsilon, \epsilon'>0$ sufficiently small, and $0<\vert c-c' \vert \ll \epsilon$, $0<\vert \delta \vert  \ll \epsilon'$. Here $B_\epsilon(v)$, $B_{\epsilon'}(v)$ are balls of radius $\epsilon$, resp., $\epsilon'$, centred at $v$ in the ambient space $\bC^N$.

The expression \eqref{def3} shows that the negative of the integer $\mu_V$ matches what, in \cite[Section 5]{BMPS}), was called the ``defect'' of the function $f-c$ on $X$ at some point of $V$.
\end{remark}

\begin{remark}\label{xsm}
In the case when $X$ is smooth, formula \eqref{nvg} reduces to an explicit topological computation of the multiplicities $n_V$ in terms of Milnor fiber and complex link information. 
Indeed, if $X$ is smooth then $\Eu_X=1_X$. Hence, the quantity $\mu_V$ appearing in  formula \eqref{nvg} becomes:
 \be\label{nv} \mu_V=\Phi_{f-c}(1_X)(V)=\chi(\widetilde{H}^*(F_V;\bC)),\ee
i.e.,  the Euler characteristic of the reduced cohomology of the Milnor fiber $F_V$ of the hypersurface $\{f=c\}$ at some point in $V$.
 If, moreover,  $f\colon X\to \C$ has only isolated singularities, then formula 
\eqref{nvg} at such an isolated singular point $P$ reduces, in view of \eqref{nv}, to $n_P=\mu_P$, the Milnor number of $f$ at $P$, as already indicated in Example \ref{str_iso}.
\end{remark}

\subsection{The number of Morse points on $X_{\rm reg}$ in a generic deformation of $f$}\label{ss:numberMorse}
In this section we show how the defects $\mu_V$ appearing in Theorem \ref{nvgen} and Remark \ref{27} can be used to compute the number of Morse critical points on the regular part of $X$ in a Morsification of $f$.

Let, as before, $X \subset \bC^N$ be a  complex affine variety, and $f\colon X \to \bC$ the (non-constant) restriction to $X$ of a polynomial function.  Let $l$ be a general linear function, and consider the deformation of $f$ given by $f_t:=f-tl$, $t \in \bC$. Chose $t\neq 0$ with $\vert t \vert$ small enough such that $f_t\colon X \to \bC$ is a holomorphic Morse function on the regular locus $X_{\rm reg}$ of $X$. Denote as before by $\# \Sing(f_t|_{X_{\reg}})$ the number of Morse critical points of $f_t$ on $X_{\reg}$. With these notations, we prove the following.
\begin{theorem}\label{morsenumber}
The number of Morse critical points on $X_{\reg}$ in a generic deformation $f_t:=f-tl\colon X \to \bC$ of $f$ is given by:
\be\label{mnu}
\# \Sing(f_t|_{X_{\reg}})=m_\infty + (-1)^{\dim X-1} \sum_{c \in \bC}\sum_{\{V \in \mathscr{S}, V \subset \Sigma_c\}} \chi(V \setminus V \cap H_t) \cdot \mu_V
\ee
where $m_\infty $ is the number of points of $\Sing(f_t|_{X_{\reg}})$ that go to infinity as $t \to 0$,  the first sum is over the critical values $c$ of $f$, for each of which we set $\Sigma_c=f^{-1}(c) \cap  \Sing_\mathscr{S} f$, $\mu_V$ is as in \eqref{def}, and $H_t:=l^{-1}(t)$ is a generic hyperplane.
\end{theorem}

\begin{proof} Assume for simplicity that $m_\infty =0$. Then we have by \eqref{count} that:
\be\label{count2} \# \Sing(f_t|_{X_{\reg}})=\sum_{V \subset \Sing_\mathscr{S} f}n_V \cdot \#\Sing(l|_{V}),\ee
with $n_V$ determined by formula \eqref{fcf} and Theorem \ref{nvgen}.

Next, using the fact that each $\overline{V}$ is affine and $H_t$ was chosen generically (hence transversal to the Whitney stratifications of $X$ and of $\overline{V}$), it follows from \cite[formula (2)]{STV1} that for each stratum $V$ appearing in \eqref{count2} we have:
\be\label{count3}
\#\Sing(l|_{V})=(-1)^{\dim V} \cdot \big[ \chi(\overline{V}, \Eu_{\overline{V}}) - \chi(\overline{V} \cap H_t, \Eu_{{\overline{V}}\cap H_t}) \big]. 
\ee
Note also that, by transversality, \be\label{restr}\Eu_{{\overline{V}}\cap H_t}=\Eu_{\overline{V}}\vert_{H_t}.\ee
It then follows from \eqref{count2} and \eqref{count3} that
\be\label{223}
\# \Sing(f_t|_{X_{\reg}})=\sum_{V \subset \Sing_\mathscr{S} f} (-1)^{\dim V} \cdot n_V \cdot 
 \big[ \chi(\overline{V}, \Eu_{\overline{V}}) - \chi(\overline{V} \cap H_t, \Eu_{{\overline{V}}\cap H_t}) \big]. 
\ee

For each critical value $c$ of $f$, applying the Euler characteristic over $\Sigma_c$ to formula \eqref{fcfb} yields
\be\label{fcfbf1}
\chi(\Sigma_c,\Phi_{f-c}(\Eu_X))=\sum_{V \subset \Sigma_c } (-1)^{\codim V-1}\cdot n_V \cdot \chi(\overline{V}, \Eu_{\overline{V}}),
\ee
On the other hand, by restricting \eqref{fcfb} to $H_t$, using \eqref{restr}, and applying the Euler characteristic over $\Sigma_c \cap H_t$, we get:
\be\label{fcfbf2}
\chi(\Sigma_c \cap H_t,\Phi_{f-c}(\Eu_X)\vert_{H_t})=\sum_{V \subset \Sigma_c } (-1)^{\codim V-1}\cdot n_V \cdot \chi(\overline{V}\cap H_t, \Eu_{\overline{V}\cap H_t}),
\ee

Combining \eqref{223}, \eqref{fcfbf1} and  \eqref{fcfbf2} yields that
\be\label{226}
\# \Sing(f_t|_{X_{\reg}})=(-1)^{\dim X-1} \sum_{c \in \bC}
\big[ \chi(\Sigma_c,\Phi_{f-c}(\Eu_X))-\chi(\Sigma_c \cap H_t,\Phi_{f-c}(\Eu_X)\vert_{H_t}) \big].
\ee
By transversality, strata of $\Sigma_c \cap H_t$ are of the form $V \cap H_t$, with $V \subset \Sigma_c$. So we can refine the stratification of $\Sigma_c$ with strata of the form $V \cap H_t$ and $V \setminus V \cap H_t$, for each $V \subset \Sigma_c$. The assertion follows now from \eqref{226} and the definition of the Euler characteristic of a constructible function (as a weighted sum over the strata of an adapted Whitney stratification).
\end{proof}

\section{The polar viewpoint on computing the local multiplicities}\label{s:polarcomp}

 We first show, in Section \ref{s:detect},  how the limit points of stratified Morse trajectories can be detected by the polar curve. Then, for functions with only isolated stratified singularities, we show in Section \ref{s:isoldecomp} how to compute the corresponding local multiplicities in terms of polar multiplicities.

\subsection{Detecting the limit points in a Morsification}\label{s:detect}

This section addresses the fully general setting of the problem of finding the trajectories 
of stratified Morse points and their asymptotic limits at the singular set of $f$. 
We will exploit the key observation that  the polar curve contains all the stratified Morse trajectories.

Let $X \subset \C^N$ be a closed irreducible affine variety and $f\colon X \to \bC$ the restriction to $X$ of a polynomial function. 
As in the Introduction, we endow $X$ with a Whitney stratification $\sW$ with finitely many strata, and let
$\Sing_{\sW}f:=\bigcup_{V\in \sW} \Sing f_{|V}$ be the stratified singular locus of $f$ with respect to $\mathscr{W}$. Let $l\colon \C^N \to \C$ be a general linear function. 
Then, for general $t \in \bC^*$, the function $f_t:=f-tl$ is a polynomial Morse function in the stratified sense on $X$, i.e., it has only non-degenerate isolated critical points on each positive-dimensional stratum of $X$.

 Consider the 
 map $(l,f) \colon X \to \C\times \C$ and its stratified singular locus $$\Sing_{\sW}(l,f) := \bigcup_{V\in \sW} \Sing (l,f)_{|V},$$ which is also a closed set due to the Whitney regularity of the stratification $\sW$.

\begin{definition}\label{d:polarcurve}
 For  a positive dimensional stratum $V\in \sW$,  we define the polar locus of the stratum $V$ by:
\[ \Gamma_{V} (l,f) :=  \overline{\Sing (l,f)_{|V} \setminus \Sing f_{|V}} \ \subset \overline{V},
 \]
where $\Sing (l,f)_{|V} := \{x\in V \mid \rank \Jac(l,f) <2 \}.$
 
 We call 
 \[ \Gamma_{\sW} (l,f) := \bigcup_{V\in \sW}\Gamma_{V} (l,f) 
 \]
 the {\it polar locus} of $f$ with respect to the stratification $\sW$ and to the linear function $l$. 
 \end{definition}
By definition,  we have that $\Gamma_{V} (l,f) = \emptyset$ for any $V\subset \Sing_{\sW} f$. Let us observe that all strata  $V\not\subset \Sing_{\sW}f$ of dimension 1 are included in the polar locus.

By  \cite{Ti-compo}, see also \cite[Polar Curve Theorem 7.1.2]{Tibar-book} or \cite[Lemma 2.5]{ST}, there is a Zariski-open subset $\Omega_{f}$ of the dual projective space
 $\check\bP^{N-1}$ such
that, for any $l\in \Omega_f$, the polar locus $\Gamma_{\sW} (l,f)$  is a pure dimensional curve or it is
empty.
We will consider a Zariski-open set $\Omega_{f}$ which has in addition the property that, for all $l\in \Omega_{f}$,
the singular set $\Sing_{\sW} l$  consists of only stratified Morse points.

The Morse critical points of the restriction of $f_t=f-tl$ to some stratum $V\in \mathscr{W}$ satisfy the equations of the polar curve $\Gamma_{V} (l,f)$. Let $\Gamma_{V, p} (l,f)$ and  $\Gamma_{\sW, p} (l,f)$ denote  the germs at $p$  of  $\Gamma_{V} (l,f)$  and of $\Gamma_{\sW} (l,f)$, respectively. 

The next result shows how to detect the points  $\Sing (l|_{V})$ which appear on the right-hand side of the limiting formula \eqref{eq_main} of Theorem \ref{thm_main}. 

\begin{proposition}\label{polarlinear}
Let $p\in V\in \sS$ such that  $V\subset \Sing_{\sW}f$ and $\dim V >0$. Let $l\in \Omega_{f}$.\\
If $p\in \Gamma_{\sW} (l,f) \cap V$ then $p\in \Sing (l_{|V})$. 
\end{proposition}
\begin{proof}
 Let $W\in \sS$,  $W\not\subset \Sing_{\sW}f$, such that  $V\subset \overline{W}$. 
 As in \S\ref{sec:2}, let   $T^*_{\overline{W}} \bC^N$ be the conormal space of $\overline{W}$ in the cotangent bundle $T^*\bC^N$, and let $\pi : T^*_{\overline{W}} \bC^N \to \overline{W}$ denote the projection. 
 
 If $p\not\in \Sing (l_{|V})$ then  $\d l \not\in \pi^{-1}(p)$. On the other hand, it is well-known that the Whitney stratification $\sS$ is also Thom (a$_{f}$)-regular at all the strata of the singular locus $\Sing_{\sW}f$, cf  \cite{BMM}, see also \cite{Ti-compo} or \cite[Theorem A1.1.7]{Tibar-book}. This means that, for the pair of strata $(W,V)$ as above,  $\d f_{|W}$ is independent of $\d l_{|W}$  in the neighborhood of $p$.
In turn, this implies, by definition,  that the polar locus $\Gamma_{W} (l,f) $ is empty in the neighborhood of $p$. Since this holds for any stratum $W$ as above, our statement is proved.
\end{proof}

\subsection{Isolated singularities: number of Morse points and polar multiplicities}\label{s:isoldecomp}

We work here under the assumption that $f$  has only stratified isolated singularities on $X$ with respect to the stratification $\sW$. 
  
Recall formula \eqref{sta}, which in our setup of $\dim Sing_{\sW}f =0$ reads as:
$$\lim_{t\to 0}\Sing(f_t|_{X_{\reg}})=\sum_{P\in Sing_{\sW}f} n_P \cdot P,
$$
with each $n_P \in \mathbb{Z}_{\geq 0}$.

The main result of this section gives an expression of the multiplicities $n_P$, for each $P\in Sing_{\sW}f$,  in terms of local polar multiplicities.

\begin{theorem}\label{t:main1}
Let $X\subset \bC^{N}$ be any affine variety and let $f\colon X\to \bC$ be the restriction of a polynomial function on $\bC^{N}$ such that it has only  isolated singularities on $X$ with respect to the stratification $\sW$. Then, for any general linear function $l$,  we have the following equality {simultaneously} at any point $P \in \Sing_{\sW} f$:
\begin{equation}\label{eq:main1}
 n_P = \mult_{P}\left( \Gamma_{X_{\reg}}(l,f), f^{-1}(f(P)) - \mult_{P}(\Gamma_{X_{\reg}}(l,f), l^{-1}(l(P)\right),
  \end{equation}
  where the intersection multiplicities of \eqref{eq:main1} are considered in the ambient space $\bC^{N}$.
\end{theorem}

 We sketch below a proof of Theorem \ref{t:main1}, starting by defining the genericity of $l$.

\begin{proof}[Proof of Theorem \ref{t:main1}]
First of all let us remark that one may choose an $l \in \Omega_{f}$ in the intersection of a finite number of Zariski-open subsets of linear forms such that $l$ is locally generic at each point of the finite set $\Sing_{\sW}f$. This will be our globally generic $l$.

Secondly, by combining formulas \eqref{niis} and \eqref{def3} with the result of \cite[Proposition 2.3 and p.281]{STV2}, we deduce that for each isolated stratified singularity $P \in \Sing_{\sW} f$, the multiplicity $n_P$ coincides with the number of Morse critical points on $X_{\reg}$ of a generic deformation $f_t=f-tl$ of $f$ which converge to $P$ as $t\to 0$. 
The local formula 
\begin{equation}\label{eq:masseyformula}
\begin{array}{l}
 \# \{ q(t)\in \Sing(f_t|_{X_{\reg}}) \mid q(t)\to P \mbox{ as } t\to 0 \} =  \\  \ \ \ \  \ \ \ \  \ \ \ \  \ \ \ \ 
  \mult_{P}\left( \Gamma_{X_{\reg}}(l,f), f^{-1}(f(P)\right) - \mult_{P}\left( \Gamma_{X_{\reg}}(l,f), l^{-1}(l(P)\right)
  \end{array}
\end{equation}
 has been proved by Massey in \cite[Theorem 3.2]{Massey2}\footnote{For a direct and radically shorter proof of this result, see our upcoming paper \cite{MT2}.}, based on ideas previously developed in works of Siersma \cite{Si-bouquet} and Tib\u{a}r \cite{Ti-bouquet}.

Then for any point $P \in \Sing_{\sW}f$, the equality \eqref{eq:main1} follows from \eqref{eq:masseyformula}.
\end{proof}

As stated after \eqref{fcfi},  the local index $n_P$ is non-negative. The next result describes when $n_P$ is positive.

\begin{proposition}
Let $f$ be a polynomial on $\bC^{N}$ such that its order at $P$ is $\ge 2$. Then $n_P>0$.
\end{proposition}

\begin{proof} The condition $\ord_{P} f \ge 2$ implies that $f\in m_{X,P}^{2}$, where  $m$ denotes the maximal ideal of the local ring $\cO_{X,P}$. Then the generic polar locus $\Gamma_{X_{\reg}}(l,f)$ contains the point $P$ and since it has dimension
1, is it a non-empty curve, see e.g. \cite[Lemma 2.5]{ST}. On the other hand the restriction map $(l,f)_{|} : \Gamma_{X_{\reg}}(l,f) \to \bC^{2}$  is one-to-one, hence the image is a curve $\Delta$ (called \emph{Cerf diagram}, or \emph{discriminant}) which has the following property: each irreducible component is tangent to the axis $x=0$, cf  \cite{Ti-lefnumber}, \cite[Corollary 3.6]{Ti-thesis}. This tangency yields a positive difference between the local intersection multiplicities at the point $q:=(l,f)(P)\in \bC^{2}$, namely  $$\mult _{q}(\Delta, \{ x=0\}) -  \mult_{q} (\Delta, \{ y=0\}) >0,$$
 which  lifts, by the one-to-one correspondence, to
\[ \mult_{P}\left( \Gamma_{X_{\reg}}(l,f), f^{-1}(f(P)) - \mult_{P}(\Gamma_{X_{\reg}}(l,f), l^{-1}(l(P)\right) >0,\]
thus we get the desired positivity in \eqref{eq:main1}.
\end{proof}

For computational purposes, let us note that as a consequence of Corollary \ref{cor:ED}, Remark \ref{r25} and Theorem \ref{t:main1}, we get the following expression for the ED degree as a linear combination of polar numbers. 

\begin{corollary}\label{cor:EDiso}
Fix a data point $(u_1,\dots,u_N)\in\mathbb{C}^N$ and an algebraic variety $X\subset\mathbb{C}^N$ with a Whitney stratification $\mathscr W$.
Assume that $f=\sum_{i=1}^N(x_i-u_i)^2$ has only isolated stratified singularities on $X$, and fix a general linear function $l$. If no points of $\Sing(f-tl|_{X_{\reg}})$ escape to infinity as $t\to 0$, 
then the Euclidean distance degree of $X$ is given by
$${\rm EDdeg}(X)=\sum_{P \subset \Sing_\mathscr{W} f}
\mult_{P}\left( \Gamma_{X_{\reg}}(l,f), f^{-1}(f(P)) - \mult_{P}(\Gamma_{X_{\reg}}(l,f), l^{-1}(l(P)\right).$$
\end{corollary}

\section{Examples}\label{s:ex}   

\bex \label{e:morse}
 The following example was considered in \cite{STV2} by making use of formula \eqref{def3}.
It is interesting to  inspect it here in view of Theorem \ref{t:main1}.  Let $X := \{ x^2 - y^2 = 0\} \times \bC \subset \bC^3$ and let $f$ be the restriction to $X$ of the function $(x,y,z) \mapsto x+ 2y + z^2$.   We have here a space of dimension $2$ which has two strata
 in the coarsest Whitney stratification $\sW$, namely: $V := \Sing X$ is the $z$-axis, and $S := X_{\reg} = X\setminus \{ x=y=0\}$. By a quick computation we see that $\Sing_{\sW} f$ is the origin point $O :=(0,0,0)$.

We choose as general linear function the projection $(x,y,z) \mapsto z$, and its restriction to $X$.  By another elementary computation, or by employing the useful Proposition \ref{polarlinear}, the polar locus $\Gamma_{S}(l,f)$ appears to be empty. Then \eqref{eq:main1} yields that $n_O= 0$. 

\eex 
 
\bex \label{e:morse2}
 In order to see the effectivity of our new method described by Theorem \ref{t:main1}, we revisit here Example 5.5 from \cite{MRW5}, where it was checked with the help of a computer.  
So let $X := \{ (x^{2} + y^{2} +x)^{2}  = x^{2}+y^{2} \} \subset \bC^{2}$ be a cardioid, and let $f :\bC^{2}\to \bC$, $(x,y)\mapsto x^{2}+y^{2}$, be the complexification of the square of the distance function with respect to the origin. We also consider a general linear function $l$.
The space $X$ is an irreducible curve with a singular point at the origin $O=(0,0)\in \bC^{2}$. The function $f$ has two critical points, the origin $O$ and $P=(-2,0)$. The Whitney stratification $\cW$ of $X$ consists of the  point $O$ and the complement $S= X_{\reg} :=X\setminus O$. It  follows, by definition and because $\dim X = 1$, that the polar curve $\Gamma_{S}(l,f)$ is $X$ itself. 

We  have $\mult_{O}(\Gamma_{S}(l,f), \{f=0\}) = \mult_{O}(X, \{f=0\})$, and we observe that the plane curves 
$X$ and $\{f=0\}$ intersect transversely at $O$  in the sense that their tangent cones are different. Therefore:
$$\mult_{O}(X, \{f=0\}) = \mult_{O}X \cdot  \mult_{0}\{f=0\} = 2\cdot 2 = 4.$$

By the genericity of $l$ we have $\mult_{O}(X, \{l=0\})  = \mult_{O}X =2$. By applying \eqref{eq:main1} we then get:
 \[
 n_O = \mult_{O}(\Gamma_{S}(l,f), \{f=0\}) - \mult_{O}(\Gamma_{S}(l,f), \{l=0\}) = 4-2=2,
 \]
which means that two Morse points of the function $f_{t}:= f - tl$ on $X_{\reg}$ converge to the origin as $t\to 0$. 

The computation is similar at the point $P$ and we quickly get:
\[ \mult_{P}(X, \{f=0\}) = \mult_{P}X \cdot  \mult_{p}\{f=0\} = 1\cdot 2 = 2.\]
The intersection at $P$ of $X$ with the general slice $l=$ constant is smoothly transversal, thus the corresponding intersection multiplicity is 1. By \eqref{eq:main1} we then get:
\[ n_P = \mult_{P}(\Gamma_{S}(l,f), \{f=0\}) - \mult_{P}(\Gamma_{S}(l,f), \{l=0\}) = 2-1=1,
\]
which means that only one Morse point of the function $f_{t}$ on $X_{\reg}$ converges to $P$ when $t\to 0$. 

We note that $f_t$ has $3$ Morse critical points on $X_{\reg}$  for any general $t$, and no Morse point escapes to infinity as $t\to 0$. In view of Corollary \ref{cor:ED} and of formula \eqref{eq:main1}, the above calculation implies that ${\rm EDdeg}(X)=3$, which confirms in this new manner a well-known result, e.g., see \cite[Example 1.1]{DHOST}.
\eex


\begin{ex}\label{str_isos}
Based on formula \eqref{niis}, we perform an explicit calculation of the multiplicity $n_V$ in the special situation when both $X$ and  $f$ have an isolated singularity at the origin. Let $n=\dim X$ and let us assume $f(0)=0$. 
Formula \eqref{fcf} becomes 
\[
(-1)^{n-1}  \Phi_{f}(\Eu_X)=n_{\{0\}} \cdot \Eu_{\{0\}}, 
\]
from which one gets
\be\label{m1}
n_{\{0\}}=(-1)^{n-1}\Phi_{f}(\Eu_X)(0).
\ee
Let $F_{f,0}$ denote the Milnor fiber of $f$ at $0$. By results of Siersma \cite{Si-bouquet}, the reduced integral homology of $F_{f,0}$ has a decomposition 
\be\label{m2}
\widetilde{H}_*(F_{f,0}) \cong  \widetilde{H}_*(\lk_X(\{0\})) \oplus \widetilde{H}_*(S^{n-1})^{\oplus{k}}, 
\ee
where $\lk_X(\{0\})$ is the complex link of the origin in $X$, and $k$ is the number of times the summand 
$\widetilde{H}_*(S^{n-1})$ appears in the decomposition. 
We use \eqref{m1} to show that $$n_{\{0\}}=k.$$ 
Let $\alpha:=\Eu_X - 1_X$, so $\alpha$ is a constructible function on $X$ supported at the origin $0$. We rewrite \eqref{m1} as:
\be\label{m3}
n_{\{0\}}=(-1)^{n-1}\left\{ \Phi_{f}(1_X)(0) + \Phi_{f}(\alpha)(0) \right\},
\ee
and we have that 
\be\label{m4} \Phi_{f}(1_X)(0)=\chi(\widetilde{H}^*(F_{f,0}))=\chi(\widetilde{H}^*(\lk_X(\{0\})) + (-1)^{n-1} \cdot k,
\ee
the last equality being a consequence of \eqref{m2}. To show that $n_{\{0\}}=k$, it suffices from \eqref{m3} and \eqref{m4} to show that 
\be\label{m5}
\Phi_{f}(\alpha)(0)=-\chi(\widetilde{H}^*(\lk_X(\{0\})).
\ee
For this, we use the following identity of constructible functions, with $\alpha$ a constructible function on $X$: \be\label{van} \Phi_{f}(\alpha)=\Psi_f(\alpha) -\alpha\vert_{f=0},\ee
and recall that the nearby cycle functor $\Psi_f$ depends only on the values of $\alpha$ outside the set $\{f=0\}$. Since in our case $\alpha$ is supported only at the origin and $f(0)=0$, we get that $\Psi_f(\alpha)=0$, and hence $\Phi_{f}(\alpha)= -\alpha\vert_{f=0}$. In particular, 
\be\label{m6}
\Phi_{f}(\alpha)(0)=-\alpha(0).
\ee
Next, to compute $\alpha(0)$, we stratify $X$ with strata $X\setminus \{0\}$ and $\{0\}$, and  express $1_{X\setminus \{0\}}$ in terms of basis $\{\Eu_X, \Eu_{\{0\}}=1_{\{0\}}\}$ of constructible functions on $X$, to get  
\be\label{m7}
1_{X\setminus \{0\}}= \Eu_X -\chi(\lk_X(\{0\})) \cdot 1_{\{0\}},
\ee
or, equivalently,
\be\label{m8}\alpha:=\Eu_X - 1_X=\left(\chi(\lk_X(\{0\})) -1\right) \cdot 1_{\{0\}}.\ee
The desired identity \eqref{m5} follows by combining \eqref{m6} and \eqref{m8}.
\end{ex}


\end{document}